\documentclass[11pt]{amsart}
\usepackage{amsmath,amsthm, amsfonts, amssymb, latexsym,verbatim,bbm,longtable,mathtools}
\usepackage{multirow}
\input xy
\xyoption{all}
\usepackage{hyperref,multirow}
\usepackage{tikz, tabularx}
\usepackage{young}
\usetikzlibrary{decorations.markings,arrows}
\usepackage{genyoungtabtikz}
\usepackage{tikz-cd}

\allowdisplaybreaks[2] \textwidth15.1cm \textheight22cm \headheight12pt \oddsidemargin.4cm
\theoremstyle{plain}
\newtheorem{theorem}{Theorem}[section]

\newtheorem{lemma}[theorem]{Lemma}
\newtheorem{proposition}[theorem]{Proposition}

\newtheorem{example}[theorem]{Example}

\newtheorem{corollary}[theorem]{Corollary}

\newtheorem{conjecture}[theorem]{Conjecture}

\newcommand{\C}{\mathbb{C}}

\newcommand{\mfg}{\mathfrak{g}}
\newcommand{\mfp}{\mathfrak{p}}

\newcommand{\mfl}{\mathfrak{l}}

\DeclareMathOperator{\mfsl}{\mathfrak{sl}}

\DeclareMathOperator{\mfb}{\mathfrak{b}}
\DeclareMathOperator{\mft}{\mathfrak{t}}

\DeclareMathOperator{\Gr}{Gr}

\DeclareMathOperator{\Fl}{Fl}
\newcommand{\Fla}{\Fl(\textbf{a})}

\DeclareMathOperator{\LInv}{LInv}
\DeclareMathOperator{\Coess}{Coess}

\begin{document}

\title{The Nash blow-up of a cominuscule Schubert variety}

\author{Edward Richmond}
\email{edward.richmond@okstate.edu}

\author{William Slofstra}
\email{weslofst@uwaterloo.ca}

\author{Alexander Woo}
\email{awoo@uidaho.edu}

\begin{abstract}
We compute the Nash blow-up of a cominuscule Schubert variety.  In particular, we show that the Nash blow-up is algebraically isomorphic to another Schubert variety of the same Lie type.  As a consequence, we give a new characterization of the smooth locus and, for Grassmannian Schubert varieties, determine when the Nash blow-up is a resolution of singularities.  We also study the induced torus action on the Nash blow-up and give a bijection between its torus fixed points and Peterson translates on the Schubert variety.
\end{abstract}

\maketitle

\section{Introduction}

The Nash blow-up of a complex algebraic variety is the parameter space of tangent spaces over its smooth locus together with the limits of tangents spaces over its singular locus.  One motivation for studying the Nash blow-up is that its tautological bundle serves as an analogue of the tangent bundle for singular varieties.  In particular, in~\cite{Ma74}, MacPherson uses the Nash blow-up to develop a theory of characteristic classes for singular varieties.  Another motivation is that, while Nash blow-ups are not smooth in general, they seem to partially resolve singularities.  In \cite{No75}, Nobile proves that for algebraic curves, iterated Nash blow-ups eventually desingularize the curve.  It is an open question whether or not iterated Nash blow-ups eventually desingularize all algebraic varieties.

The main result of this paper is an explicit calculation the Nash blow-up of a Schubert subvariety of a cominuscule flag variety.  In particular, we show that the Nash blow-up is isomorphic to another Schubert variety in a generalized flag variety of the same Lie type.  Cominuscule flag varieties come in four infinite families plus two exceptional cases:
\begin{itemize}
\item the Grassmannian of $k$-dimensional subspaces in $\mathbb{C}^n$,
\item the Lagrangian Grassmannian,
\item the maximal even orthogonal Grassmannian,
\item quadric hypersurfaces, and
\item two exceptional cases of type $E_6$ and $E_7$.
\end{itemize}
As a corollary, we give a new characterization the smooth/singular locus of a cominuscule Schubert variety.  Moreover, for Schubert varieties in the Grassmannian, we determine when the Nash blow-up is a resolution of singularities.  We remark that alternate descriptions of the smooth/singular locus of cominuscule Schubert varieties have been given in \cite{BP99, Pe09, Ro14}.

In \cite{CK03}, Carrell and Kuttler (inspired by work of Peterson) study chains of curves in the Nash blow-up of torus-stable subvarieties in the flag variety with the goal of characterizing their singular loci.  For Schubert varieties, they characterize the smooth torus-fixed points as follows.  Any chain of torus-stable curves starting at the fixed point of the Schubert cell induces a chain of curves leading to a torus-fixed point of the Nash blow-up, and Carrell and Kuttler show that the singular locus of a Schubert variety is the closure of the Borel orbits of the torus-fixed points which have such curves leading to different points in the fiber over the given point.  They make these paths computationally explicit via a process we call combinatorial Peterson translation.  A consequence of our results is that Peterson translation finds all the torus-fixed points of the Nash blow-up of a cominuscule Schubert variety, not just enough of them to distinguish the singular locus.  It would be interesting to know if Peterson translation, or some other combinatorial procedure, finds all the torus-fixed points for Nash blow-ups of arbitrary Schubert varieties.

In \cite{Ma74}, MacPherson defines the Chern--Schwartz--Macpherson (CSM) class and (following unpublished work of Mather) the Chern--Mather (CM) class for a variety.  These classes are analogues of the total Chern class (of the tangent bundle) for a smooth variety.  Aluffi and Mihalcea give a formula for the CSM class of any Schubert variety \cite{AM09, AM16} and later, with Schuermann and Su, prove that these classes expand positively in terms of Schubert classes \cite{AMSS17}.  For CM classes, Jones provides an algorithm for computing them in the case of Schubert varieties of Grassmannians \cite{Jo10}.  Note that finding a combinatorially positive formula for these classes remains an open problem.  While MacPherson's original construction of these classes uses the Nash blow-up, the techniques used by the authors above instead make use of certain resolutions of Schubert varieties to compute these classes.  In the case of Grassmannian Schubert varieties, Aluffi and Mihalcea use the Bott-Samelson resolution to compute CSM classes, and Jones uses the Zelevinsky resolution for CM classes.  As another corollary of our work, we give an explicit relationship between the Nash blow-up of a Grassmannian Schubert variety and its Zelevinsky resolutions.  More precisely, we show the Nash blow-up is isomorphic to a fiber product of two particular Zelevinsky resolutions and conjecture that a similar result holds for the more general class of covexillary Schubert varieties.

We structure this paper as follows.  In Section \ref{S:Prelim}, we give preliminaries on Nash blow-ups and prove our main structure theorem on the Nash blow-ups of cominuscule Schubert varieties.  In Section \ref{S:P_translates}, we discuss Peterson translation on Schubert varieties and show how these translates are bijectively related to the torus-fixed points of the Nash blow-up.  Finally, in Section \ref{S:type_A_case}, we go over specifics for type A Grassmannian Schubert varieties.  In particular, we describe the Nash blow-up as a configuration space and characterize when it is a resolution of singularities.  We also show how the Nash blow-up is related to the Zelevinsky resolution of Grassmannian Schubert varieties and conjecture that a similar relationship holds for covexilliary Schubert varieties.

\subsection*{Acknowledgements}  We would like to thank ICERM for their hospitality in hosting our Collaborate@ICERM focus group in August 2017 from which this project originated.  ER was supported by a Young Investigators NSA grant H98230-16-1-0014. AW was supported by Simons Collaboration grant 359792. We thank Shrawan Kumar, Jenna Rajchgot, and Jesper Thomsen for helpful correspondence and Gary Kennedy and Leonardo Mihalcea for pointers to references.  We especially thank Jim Carrell for helpful conversations and insights on Peterson translation.

\section{Identifying Nash blow-ups}\label{S:Prelim}
Let $X$ be a $d$-dimensional complex projective algebraic variety
(by which we mean a reduced, irreducible scheme).
The {\bf Nash blow-up} of $X$ is defined as follows.
Let $Y$ be a smooth projective variety with an embedding $X\rightarrow Y$, and let $T_*Y$ denote the tangent bundle of $Y$.  Associated to the tangent bundle $T_*Y$ is the Grassmann bundle $\mathcal{G}(d,T_*Y)$
whose fiber over a point $y\in Y$ is the Grassmannian $\Gr(d,T_y Y)$
of $d$-dimensional subspaces of the tangent space $T_y Y$.  The Grassmann bundle can be realized as the subvariety of decomposable forms in the $d$-th exterior power $\wedge^d T_* Y$ of the tangent bundle.
Let $\pi:\mathcal{G}(d,T_*Y)\rightarrow Y$ denote the canonical projection.

Let $X^\mathrm{reg}$ denote the smooth locus of $X$.  Given $x\in X^\mathrm{reg}$, we have a $d$-dimensional
tangent space $T_x X\subseteq T_x Y$.
Define a map $\phi: X^\mathrm{reg}\rightarrow \mathcal{G}(d,T_*Y)$ by $\phi(x):=T_x X$.
Note that $\phi$ is a section of $\pi$ over $X^\mathrm{reg}$.
Now let $\widehat{X}:=\overline{\phi(X^\mathrm{reg})}$ be
the closure of the image of $\phi$.  Since $X^\mathrm{reg}$
is dense in $X$, the image $\pi(\widehat{X})=X$.
The variety $\widehat{X}$ with the canonical map $\pi:\widehat{X}\rightarrow X$ is known as the Nash blow-up of $X$.  In \cite{No75}, Nobile shows that $\widehat{X}$ is independent of the choice of embedding $X\rightarrow Y$ and hence canonical to $X$.  Note that it is not necessary to take all of $X^\mathrm{reg}$ in the above construction; any dense subset of $X^\mathrm{reg}$ suffices.

\subsection{The Nash blow-up of Schubert varieties}

Let $G$ be a connected semi-simple Lie group over $\C$. Let $T\subset B\subset P$ be a maximal torus, Borel, and parabolic subgroup of $G$.  Let $L,U\subseteq P$ denote the Levi subgroup and unipotent radical of $P$.  Let $W_L=W_P\subset W$ be the Weyl groups of $P$ and $G$, and let $W^P\simeq W/W_P$ denote the set of minimal length coset representatives.  We will use Gothic letters $\mfg,\mfp,\mfb,\mft,\mfl$ to denote the Lie algebras of the corresponding groups $G,P,B, T$ and $L$.  Let $R=R^+\sqcup R^-$ and $R_{L}=R^+_{L}\sqcup R^-_{L}$ denote the sets of roots, positive roots and negative roots for $G$ and $L$ respectively, and let $\Delta$ and $\Delta_{L}$ denote the sets of positive simple roots for $G$ and $L$.  If $\alpha\in R$, let $\mfg_{\alpha}$ denote the corresponding root space in $\mfg$.  The tangent space at the identity of the flag variety $G/P$ naturally identifies with
$$T_{eP}(G/P)=\mfg/\mfp=\bigoplus_{\alpha\in R^+\setminus R^+_{L}} \mfg_{-\alpha}.$$
The tangent bundle identifies as $$T_*(G/P)=G\times_P \mfg/\mfp$$
and the Grassmann bundle
$$\mathcal{G}(d,T_*(G/P))=G\times_P \Gr(d,\mfg/\mfp).$$
For any $w\in W^P$, define the (shifted) Schubert cell
$$\Lambda_w^P:=w^{-1}BwP/P\subset G/P.$$  The Schubert variety is defined as the closure $\overline\Lambda_w^P$.  Observe that $\Lambda_w^P$ is a dense set in the smooth locus of $\overline\Lambda_w^P$ and $eP\in\Lambda_w^P$. Let
$$T_w:=T_{eP}(\Lambda_w^P)=\bigoplus_{\alpha\in (R^+\setminus R^+_{L})\cap w^{-1}(R^-)} \mfg_{-\alpha}$$ denote the tangent space of $\Lambda_w^P$ (equivalently $\overline\Lambda_w^P$) at $eP$.  If $\dim(\Lambda_w^P)=d$, then the map $$\phi:\Lambda_w^P\rightarrow \mathcal{G}(d,T_*(G/P))$$ is given by $\phi(gP)=(g,T_w).$  The Nash blow-up of the Schubert variety $\overline\Lambda_w^P$ is
\begin{equation}\label{Eq:Nash_blowup}
\overline{\phi(\Lambda_w^P)}=\overline{w^{-1}BwP\times_P T_w}=\overline{w^{-1}BwP}\times_P \overline{P\cdot T_w}.
\end{equation}

\subsection{The Nash blow-up of cominuscule Schubert varieties}

If $P$ is a maximal parabolic subgroup, then there is a unique $\alpha\in\Delta\setminus\Delta_L$.  In this case, we say that $\alpha$ is the simple root associated with the maximal parabolic $P$.  A maximal parabolic subgroup $P$ is {\bf cominuscule} if its associated simple root appears with coefficient one in the unique highest root of $G$.  We say a flag variety $G/P$ (respectively a Schubert variety $\overline\Lambda_w^P$) is {\bf cominuscule} if $P$ is cominuscule.  A parabolic $P$ is cominuscule if and only if its unipotient radical $U$ is abelian; a proof is given by Richardson, R\"ohrle, and Steinberg \cite[Lemma 2.2]{RRS92}.    We now state and prove the main theorem of this paper.

\begin{theorem}\label{T:main1}
Let $w\in W^P$ and $\overline\Lambda_w^P$ be a cominuscule Schubert variety.  Then the Nash blow-up $\overline{\phi(\Lambda_w^P)}$ is a Schubert variety.  In particular,
$$\overline{\phi(\Lambda_w^P)}\simeq \overline{w^{-1}BwQ}/Q$$ for the standard parabolic subgroup $Q\subseteq P$, where $Q$ is generated by the set of simple roots
$$\Delta_w:=\{\beta\in\Delta_{L}\mid w(\beta)\in\Delta\}.$$ Equivalently, $Q=BW_{\Delta_w}B$.
\end{theorem}

\begin{proof}
Consider the Levi factorization of $P=L\cdot U$. Since $P$ is cominuscule, the unipotent radical $U$ is abelian and hence acts trivially on $\mfg/\mfp$.  Furthermore, $U$ acts trivially on the Grassmannian $\Gr(d,\mfg/\mfp).$  Hence the orbit $P\cdot T_w=L\cdot T_w$ in $\Gr(d,\mfg/\mfp).$  If $Q'$ denotes the stabilizer of $T_w$ in $L$, then $L\cdot T_w=L/Q'$.  Since $w\in W^P$, the Borel subgroup $B_L:=B\cap L\subseteq w^{-1}Bw$, and thus $T_w$ is a $B_L$-module.  Hence $B_L\subseteq Q'$, so $Q'$ is a standard parabolic subgroup of $L$.  Let $Q$ be the standard parabolic subgroup of $P$ for which $Q'=L\cap Q$ and hence $P/Q\simeq L/Q'$.  Thus $P\cdot T_w=P/Q$, and since $P/Q$ is closed, we have $\overline{P\cdot T_w}=P/Q$.  The first part of the theorem is proved by the following calculation using Equation \eqref{Eq:Nash_blowup}:
\begin{align*}
\overline{\phi(\Lambda_w)}=\overline{w^{-1}BwP}\times_P \overline{P\cdot T_w}\simeq \overline{w^{-1}BwP}\times_P P/Q=\overline{w^{-1}BwQ}/Q.
\end{align*}

For any $\beta\in R$, let $U_\beta$ denote the root subgroup of $G$ corresponding to the root space $\mfg_\beta$.  To prove the second part of the theorem, we will show that $\beta\in \Delta_w$ if and only if $U_\beta$ stabilizes
$$T_w=\bigoplus_{\alpha\in R^+\setminus R^+_{L}\cap w^{-1}R^-} \mfg_{-\alpha}.$$
First note that $U_\beta$ stabilizes $T_w$ if and only if, for all $\alpha\in R^+\setminus R^+_{L}\cap w^{-1}R^-$, we have $U_\beta\cdot \mfg_{-\alpha}\subseteq T_w$.  Passing to the Lie algebra, this inclusion holds if and only if $[\mfg_{-\beta},\mfg_{-\alpha}]\subseteq T_w$.  In other words, $U_\beta\cdot \mfg_{-\alpha}\subseteq T_w$ if and only if $\alpha+\beta\in  R^+\setminus R^+_{L}\cap w^{-1}R^-$ whenever $\alpha\in  R^+\setminus R^+_{L}\cap w^{-1}R^-$ and $\alpha+\beta\in R^+\setminus R^+_{L}$.  Hence it suffices to show this last property holds for $\beta\in\Delta_{L}$ precisely when $w(\beta)$ is simple.

Suppose $w(\beta)$ is simple.  Since $w$ is a minimal coset representative and $\beta\in R^+_{L}$, the root $w(\beta)$ is also positive.  If $\alpha\in R^+\setminus R^+_{L}\cap w^{-1}R^-$, then $w(\alpha)\in R^-$.  Furthermore, if $\alpha+\beta\in R$, then  $w(\alpha+\beta)=w(\alpha)+w(\beta)\in R^-$ since it is the sum of a negative root and a positive simple root.

Now suppose $w(\beta)$ is not simple.  Then $w(\beta)=w(\alpha)+w(\alpha')$ for some roots $\alpha,\alpha'$ where $w(\alpha),w(\alpha')\in R^+$.  Since $\beta$ is simple and $\beta=\alpha+\alpha'$, either $\alpha$ or $\alpha'$ is negative; without loss of generality, assume $\alpha\in R^-$ and $\alpha'\in R^+$.  Since $w$ is a minimal coset representative, $\alpha\in R^-$, and $w(\alpha)\in R^+$, we must have $\alpha\in R\setminus R_{L}$ and hence $\alpha'\in R\setminus R_{L}$.  Otherwise, we would have $\alpha=\beta-\alpha'\in R_{L}$.  We now have that $-\alpha\in R^+\setminus R^+_{L}\cap w^{-1}R^-$ and $\alpha'=-\alpha+\beta\not\in R^+\setminus R^+_{L}\cap w^{-1}R^-$.  Hence $[\mfg_{-\beta},\mfg_{\alpha}]\not\subseteq T_w$ even though $\mfg_{\alpha}\subseteq T_w$.
\end{proof}

One immediate consequence of Theorem \ref{T:main1} is that the Nash blow-up of a cominuscule Schubert variety is normal and Cohen-Macaulay.  Another consequence is a new characterization of the smooth locus.  Let $\leq$ denote the Bruhat partial order on the Weyl group $W$ and let $[e,w]$ denote Bruhat interval $\{y\in W\mid e\leq y\leq w\}$.  Define the (unshifted) Schubert variety $$X_w^P:=\overline{BwP}/P=\bigcup_{v\in [e,w]\cap W^P} BvP/P.$$

\begin{corollary}\label{C:smooth_locus}
Let $P$ be cominuscule with $w\in W^P$ and $v\in [e,w]\cap W^P$.  Let $Q$ be generated by $\Delta_w$ as in  Theorem \ref{T:main1}.  Then $vP$ is a smooth point of $X_w^P$ if and only if $$vW_P\cap [e,w]\subseteq vW_Q.$$
\end{corollary}

\begin{proof}
By Theorem \ref{T:main1}, we have the Nash blow-up $\pi: X_w^Q\rightarrow X_w^P$. Now $vP$ is a smooth point of $X_w^P$ if and only if $|\pi^{-1}(vP)|=1$.  By \cite[Lemma 4.6]{RS16}, we have that
$$\pi^{-1}(vP)=v\bigcup BuQ/Q,$$ where the union is over $u\in W^Q\cap W_P$ such that $vu\in [e,w]$.  This fiber is a single point precisely when $v(W^Q\cap W_P)\cap [e,w]=\{v\}$, which is equivalent to $vW_P\cap [e,w]\subseteq vW_Q$.
\end{proof}

We remark that alternate descriptions of the smooth/singular locus of cominscule Schubert varieties have been given in \cite{BP99, Pe09, Ro14}.  We do not know of a direct combinatorial proof connecting Corollary \ref{C:smooth_locus} to any of these characterizations.


\section{Combinatorial Peterson translates}\label{S:P_translates}

Let $X_w^P$ be an arbitrary Schubert variety of dimension $d$, and suppose $x,y\in X_w^P$ are the unique $T$-fixed points in some $T$-invariant curve $C\subseteq X_w^P$.  Then there exists $\alpha\in R^+$ such that $r_\alpha x=y$ and $C=\overline{U_\alpha x}$, where $r_\alpha\in W$ is the reflection and $U_\alpha$ is the root subgroup associated to $\alpha$.  Geometrically, {\bf Peterson translation}, defined in \cite{CK03} by Carrell and Kuttler, is the act of translating the Zariski tangent space $T_x(X_w^P)$ along the open curve $C\setminus\{y\}$.  Taking the limit gives a subspace of $T_y(X_w^P)$ of dimension $\dim T_x(X_w^P)$.  If $x$ is a smooth point, this process produces a curve in the Nash blow-up of $X_w^P$ connecting the tangent space of $x$ and a point in the fiber over $y$.  If $x$ is not smooth, then we can apply the same process to any $d$-dimensional subspace $M\subseteq T_x(X_w^P)$ to produce a curve in $\mathcal{G}(d,T_*(G/P))$ connecting $M$ and a point in the fiber over $y$.  If $M$ is in the Nash blow-up of $X_w^P$, then the resulting curve will also be in the Nash blow-up.  For any subspace $M\subseteq T_x(X_w^P)$ and $\alpha\in R^+$, define $$\tau_\alpha(x,M)\subseteq T_{r_\alpha x}(X_w^P)$$ to be the Peterson translate of $M$ along the root $\alpha$.

In the next subsection, we describe a combinatorial version of the map $\tau_{\alpha}$, given in \cite{CK03}, for Schubert varieties in an arbitrary $G/P$.  In this case, the $T$-fixed points $x$ are elements of the Weyl group $W$.  Furthermore, we only consider subspaces $M\subseteq T_x(X^P_w)$ that are $T$-stable, and we represent a subspace $M$ by the weights (which will be roots) of the $T$-action.  We will then give, for $P$ cominuscule and $w\in W^P$, an explicit bijection between the $T$-fixed points in the Nash blow-up of $X_w^P$ and the combinatorial Peterson translates on $X_w^P$.

\subsection{Combinatorial Peterson translation}
Fix a Weyl group $W$ with root system $R$ and a standard parabolic subgroup $W_P\subseteq W$ with root system $R_L\subseteq R$.  (We use the Levi subgroup $L\subseteq P$ to index $R_L$.)  For an element $w\in W$, the {\bf left inversion set} of $w$ is
$$\LInv(w):=w(R^-)\cap R^+.$$  Furthermore, define the {\bf $P$-left inversion set} of $w$ as
$$\LInv^P(w):=w(R^-\setminus R^-_L)\cap R^+.$$  Note that, in general, $\LInv^P(w)\subseteq \LInv(w)$, with equality if and only if $w\in W^P$.  Furthermore, if $w=vu$ is the parabolic decomposition of $w$, with $v\in W^P$ and $u\in W_P$, then $\LInv^P(w)=\LInv^P(v)$, since $u(R^-\setminus R^-_L)=R^-\setminus R^-_L$.

Given an element $w\in W^P$, a subset $M\subseteq w(R^-\setminus R^-_L)$, and a root $\alpha\in \LInv(w)=\LInv^P(w)$, the {\bf combinatorial Peterson translate} of $M$ from $w$ by $\alpha$, denoted $\tau_\alpha(w, M)$, is defined as follows:

An {\bf $\alpha$-string} in $w(R^-\setminus R_L^-)$ is an equivalence class of $w(R^-\setminus R_L^-)$ under the equivalence relation where $\beta\sim\beta'$ if $\beta-\beta'$ is a multiple of $\alpha$.  We denote the $\alpha$-string containing $\beta$ by $[\beta]_{\alpha,w}$.  Any $\alpha$-string $[\beta]_{\alpha,w}$ has a unique element $\mu_{[\beta]_{\alpha,w}}$ such that $\mu_{[\beta]_{\alpha,w}}-\alpha\not\in w(R^-\setminus R_L^-)$.  We call $\mu$ the {\bf $\alpha$-minimal} element of the $\alpha$-string.

Given a subset $M\subseteq w(R^-\setminus R_L^-)$, define the $\alpha$-shift of $M$, denoted $\sigma_\alpha(w,M)$, to be
$$\sigma_\alpha(w,M):=\bigcup_{\beta\in M} \{\mu_{[\beta]_{\alpha,w}}+k\alpha \mid 1\leq k\leq \#([\beta]_{\alpha,w}\cap M)\}.$$  In other words, $\sigma_\alpha(w,M)$ is the result of shifting every $\alpha$-string in $M$ as far towards its $\alpha$-minimal element as possible.  We define the combinatorial Peterson translate
$$\tau_\alpha(w,M):=(\widetilde{r_\alpha w}, r_\alpha(w, \sigma_\alpha(M))),$$
where $\widetilde{r_\alpha w}$ denotes the minimal coset representative for $r_\alpha w$ in $r_\alpha wW_P$.  Here, $r_\alpha$ is considered first as an element of $W$ and second as a reflection acting on each root in  $\sigma_\alpha(M)\subseteq R$.  As mentioned at the beginning of the section, if $vP$ is a $T$-fixed point in the Schubert variety $X_w^P$ and $M\subseteq v(R^-\setminus R_L^-)$ is the set $T$-weights of some subspace of $T_{vP}(X_w^P)$, then $\tau_\alpha(v,M)$ will be the set of  $T$-weights of the subspace in $T_{r_\alpha vP}(X_w^P)$ obtained by Peterson translating along the root $\alpha$.


Given an element $w\in W^P$, we say $(z, N)$ is an {\bf eventual Peterson translate} of $w$ (with respect to $P$) if there exists a sequence
$$(w,\LInv(w))=(z_0, M_0)\xrightarrow[\tau_{\gamma_1}]{}(z_1, M_1)\xrightarrow[\tau_{\gamma_2}]{}  \cdots  \xrightarrow[\tau_{\gamma_k}]{}(z_k, M_k)=(z,N)$$
where $(z_i, M_i)=\tau_{\gamma_i}(z_{i-1}, M_{i-1})$ for some $\gamma_i\in\LInv(z_{i-1})$.
Geometrically, eventual Peterson translation corresponds to taking the tangent space at the unique $T$-fixed point in the Schubert cell and translating along a sequence of $T$-stable curves in the Nash blow-up. In \cite[Theorem 1.3]{CK03}, Carrell and Kuttler prove that $uP$ is a singular point of $X_w^P$ if and only if there exists $v\geq u$ such that there are two eventual Peterson translates $(v,N)$ and $(v,N')$ of $w$ with $N\neq N'$.


Our goal in this section is to prove the following theorem.  In general, we would like to know if the same result holds for arbitrary $P$ and $w$.

\begin{theorem}\label{T:main2}
Let $G/P$ be a cominuscule flag variety, and let $w\in W^P$.  Then there is a bijection between the set of $T$-fixed points of the Nash blow-up of $X_w^P$ and the eventual Peterson translates of $w$ (with respect to $P$).

If $Q$ is the parabolic generated by $\Delta_w$ as in  Thoerem \ref{T:main1}, then this bijection is given explicitly by sending $v\in W^Q\cap [e,w]$ to $(\widetilde{v}, vw^{-1}(\LInv^P(w))$.
\end{theorem}

\begin{example}
Let $G$ of be type $A_3$ and $w=s_1s_3s_2$ with $P$ corresponding to $\{s_1,s_3\}\subset\Delta$.  In this case, we have $\Delta_w=\emptyset$ and hence $Q$ is the Borel subgroup $B$.  The left inversion set of $w$ is $$\LInv(w)=\{\alpha_1,\alpha_3,\alpha_1+\alpha_2+\alpha_3\},$$ and the $\alpha$-string shifting action of $\sigma_{\alpha}$ is always trivial.  The bijection between the eventual Peterson translates of $w$ and the $T$-fixed points of $X_w^Q$ is given in the following table.

$$
\begin{tabular}{|c|c|c|}
\hline
$T$-fixed pts of $X_w^Q$ & \multicolumn{2}{c|}{Eventual Peterson translates of $X_w^P$,\ $(\widetilde{v},N)$} \\ \hline
v & $\widetilde{v}$ &  $N$ \\ \hline
$s_3s_1s_2$ &$s_3s_1s_2$    &$\{\alpha_1,\alpha_3,\alpha_1+\alpha_2+\alpha_3\}$  \\
$s_3s_2$    & $s_3s_2$      &$\{-\alpha_1,\alpha_3,\alpha_2+\alpha_3\}$ \\
$s_1s_2$    &$s_1s_2$       &$\{\alpha_1,-\alpha_3,\alpha_1+\alpha_2\}$ \\
$s_2$       &$s_2$          &$\{-\alpha_1,-\alpha_3,\alpha_2\}$  \\
$s_1s_3$    &$e$            &$\{-(\alpha_1+\alpha_2),-(\alpha_2+\alpha_3),-(\alpha_1+\alpha_2+\alpha_3)\}$ \\
$s_3$       &$e$            &$\{-\alpha_2,-(\alpha_2+\alpha_3),-(\alpha_1+\alpha_2+\alpha_3)\}$\\
$s_1$       &$e$            &$\{-\alpha_2,-(\alpha_1+\alpha_2),-(\alpha_1+\alpha_2+\alpha_3)\}$ \\
$e$         &$e$            &$\{-\alpha_2,-(\alpha_1+\alpha_2),-(\alpha_2+\alpha_3)\}$ \\ \hline
\end{tabular}
$$

We further illustrate the sequential structure of eventual Peterson translation with the following graph.
$$
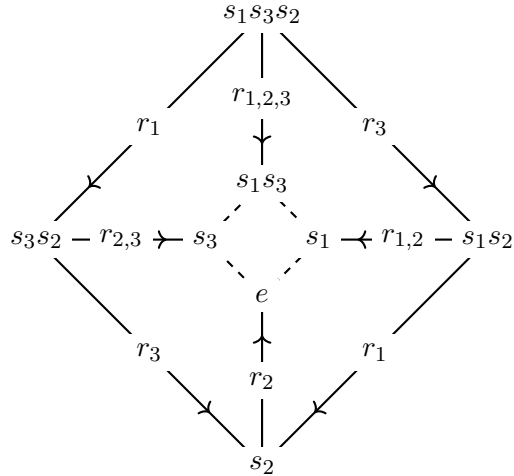
\begin{tikzpicture}[scale=1.5,decoration={markings,mark=at position 0.77 with {\arrow{>}}}]
\def \radius {28pt}
\draw[thick,postaction={decorate}] (0,4)--(-2,2) node [midway, fill=white]{$r_{1}$};
\draw[thick,postaction={decorate}] (0,4)--(2,2)node [midway, fill=white]{$r_3$};
\draw[thick,postaction={decorate}] (0,4)--(0,2.5)node [midway, fill=white]{$r_{1,2,3}$};
\draw[thick,postaction={decorate}](-2,2)--(0,0)node [midway, fill=white]{$r_{3}$};
\draw[thick,postaction={decorate}](-2,2)--(-0.5,2)node [midway, fill=white]{$r_{2,3}$};
\draw[thick,postaction={decorate}](2,2)--(0,0)node [midway, fill=white]{$r_{1}$};
\draw[thick,postaction={decorate}](2,2)--(0.5,2)node [midway, fill=white]{$r_{1,2}$};
\draw[thick,postaction={decorate}](0,0)--(0,1.5)node [midway, fill=white]{$r_{2}$};
\draw[dashed, thick](0,1.5)--(0.5,2)--(0,2.5)--(-0.5,2)--(0,1.5);
\node[fill=white] at (0,4) {$s_1s_3s_2$};
\node[fill=white] at (-2,2) {$s_3s_2$};
\node[fill=white] at (0,0) {$s_2$};
\node[fill=white] at (2,2) {$s_1s_2$};
\node[fill=white] at (0,2.5) {$s_1s_3$};
\node[fill=white] at (-0.5,2) {$s_3$};
\node[fill=white] at (0,1.5) {$e$};
\node[fill=white] at (0.5,2) {$s_1$};
\end{tikzpicture}$$
The vertices are the $T$-fixed points of the Nash blow-up, directed paths from $w=s_1s_3s_2$ correspond to eventual Peterson translates, and the edge labels show the reflections by which the translation occurs.  For notational simplicity we denote the reflection $r_{\alpha}$ using the indices of the simple roots that appear in $\alpha$.  (For example $r_{1,2}:=r_{\alpha_1+\alpha_2}$.)  Note that if two paths terminate at the same vertex, then the corresponding eventual Peterson translates are the same.  The dashed edges connect the $T$-fixed points over the identity $eP$.
\end{example}

\subsection{Proof of Theorem \ref{T:main2}}

Before we prove Theorem \ref{T:main2}, we need the following lemmas about Bruhat order on $W$.  Recall that if a Bruhat interval $[z,x]$ has length two, then there are exactly two elements $y, \tilde y$ for which $z<y,\tilde y<x$ \cite[Lemma 2.7.3]{BB}.

\begin{lemma}\label{L:length_2_interval}
Let $[z,x]=\{z<y,\tilde y<x\}$ be a length two Bruhat interval in $W$ such that $z=r_\alpha y=r_\gamma \tilde y$ and $x = r_\beta y=r_\delta \tilde y$ for some (not necessarily simple) reflections $r_\alpha,r_\beta,r_\delta,r_\gamma$.
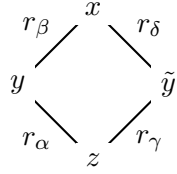
$$\begin{tikzpicture}
\draw[thick](0,0)--(1,1)--(0,2)--(-1,1)--(0,0);
\node at (-0.75,0.25){$r_{\alpha}$};
\node at (-0.75,1.75){$r_{\beta}$};
\node at (0.75,0.25){$r_{\gamma}$};
\node at (0.75,1.75){$r_{\delta}$};
\node[fill=white] at (0,2) {$x$};
\node[fill=white] at (-1,1) {$y$};
\node[fill=white] at (0,0) {$z$};
\node[fill=white] at (1,1) {$\tilde y$};
\end{tikzpicture}$$
Then $r_\delta$ and $r_\gamma$ are in the parabolic subgroup generated by $r_\alpha$ and $r_\beta$.
\end{lemma}

\begin{proof}
Define the parabolic subgroup $W':=\langle r_\alpha, r_\beta\rangle$ with induced Bruhat order $\leq'$ and pattern map $\phi:W\rightarrow W'$ defined in \cite[Section 2]{BB03} (This map appears as the flattening map in \cite{BP05}).  Since $r_\beta r_\alpha=r_\delta r_\gamma$, it suffices to show $r_\delta\in W'$.  Theorem 2 of \cite{BB03} implies that
$\phi(x)=r_\beta\phi(y)$ and $\phi(y)=r_\alpha\phi(z).$
Since $W'$ is a rank 2 Coxeter group, we have that $\phi(z)\leq'\phi(y)\leq'\phi(x)$ and hence there exists $y'\in W'$ such that $\phi(z)\leq'y'\leq'\phi(x)$ with $\phi(y)\neq y'$.  Writing $y'=r'\phi(x)=\phi(r'x)$ for some $r'\in W'$ implies $z\leq r'x\leq x$.  Since $r'x\neq y$, we must have that $r'x=\tilde y=r_\delta x$ and hence $r_\delta\in W'$.
\end{proof}

\begin{lemma}\label{L:P-Bruhat_structure}
Suppose $w\in W^P$ and $z<w$.  Then there exists $\tilde y\leq w$ with $z=r_\gamma \tilde y$ for $\gamma\in \LInv^P(\tilde y)$ and $\ell(\tilde y)=\ell(z)+1$.
\end{lemma}

\begin{proof}
We will prove the lemma by induction on $\ell(w)-\ell(z)$.  In the base case where $\ell(w)-\ell(z)=1$, the lemma holds since $\LInv^P(w)=\LInv(w)$.

Suppose $\ell(w)-\ell(z)>1$.  By the definition of Bruhat order, we know that there exists $y\leq w$ with $y$ covering $z$, so $z=r_\alpha y$ for some $\alpha\in\LInv(y)$.  We may assume that $\alpha\not\in\LInv^P(y)$; otherwise we are finished.  By the inductive hypothesis, there exists $x\leq w$ with $x$ covering $y$ such that  $y=r_\beta x$ for some $\beta\in\LInv^P(x)$.  Now, by~\cite[Lemma 2.7.3]{BB} there exists a unique $\tilde y\neq y$ with $z<\tilde y<x$.  Let $\gamma$ and $\delta$ be the roots such that $\tilde y=r_\delta x=r_\gamma z$.  (See the diagram in Lemma \ref{L:length_2_interval}.)


Note that $\alpha\in y(R_L)$, while $\beta\not\in x(R_L)$.  Hence, $-\alpha=r_\alpha \alpha \in z(R_L)$, while $r_{\alpha}(-\beta)\not\in z(R_L)$.  By Lemma \ref{L:length_2_interval}, $r_\gamma$ is in the parabolic subgroup generated by $r_\alpha$ and $r_\beta$, which implies the root $\gamma$ is in the span of the roots $\alpha$ and $\beta$.  Furthermore, $-\gamma$ is not a multiple of $\alpha$, as $\tilde y\neq y$.  Therefore, $-\gamma\not\in z(R_L)$ as $r_\alpha(-\beta)$ is a linear combination of $\alpha$ and $\gamma$, $r_\alpha(-\beta)\not\in z(R_L)$, and the set $z(R_L)$ is closed under negation. Hence, $\gamma\not\in \tilde y(R_L) = r_\gamma z(R_L)$.  Since $\gamma\in\LInv(\tilde y)$, we have $\gamma\in\LInv^P(\tilde y)$.  
\end{proof}

Lemma \ref{L:P-Bruhat_structure} can be rephrased in terms of the $P$-Bruhat order defined by Bergeron and Sottile.  It states that, if $w\in W^P$ and $z\leq w$ in ordinary Bruhat order, then $z\leq w$ in $P$-Bruhat order.  If $P$ is a cominuscule parabolic in type A or C, the lemma follows from concrete descriptions of $P$-Bruhat order found respectively in \cite[Theorem 1.1.2]{BS98} and \cite[Proposition 2.5]{BS02}.

\begin{proof}[Proof of Theorem \ref{T:main2}]
Since $P$ is cominuscule, for any roots $\alpha,\beta\in R^-\setminus R^-_L$, we have that $\alpha+\beta\not\in R^-\setminus R^-_L$.  Hence each $\alpha$-string in $z(R^-\setminus R^-_L)$ consists of only one root for any $z\in W$ and $\alpha\in z(R^-\setminus R^-_L)$.  Therefore, $\sigma_\alpha(z,M)=M$ for any $M\subseteq z(R^-\setminus R^-_L)$, and $\tau_\alpha(z,M)=(\widetilde{r_\alpha z}, r_\alpha(M))$.

Let $E:=w^{-1}(\LInv(w))=(R^-\setminus R_L^-)\cap w^{-1}(R^+)$, and define the map
$$T:\{z\in W^Q\mid z\leq w\}\rightarrow \{(v,M)\mid v \in W^P\ \text{and}\ M\subseteq v(R^-\setminus R_L^-)\}$$
by $T(z):=(\widetilde{z},z(E))$, where $\widetilde{z}$ is the minimal length representative of the coset $zW_P$.  We show that $T$ is an injection, that $T$ maps into the codomain as claimed, and that its image coincides with the set of eventual Peterson translates of $w$.


To show that $T$ is an injection, it suffices to show that the stabilizer of $E$ under the action of $W$ on $R$ is $W_Q$.  This fact holds since $Q$ is the stabilizer of $\bigoplus_{\alpha\in E} \mathfrak{g}_\alpha$ for the action of $L$ on $\Gr(d,\mathfrak{g}/\mathfrak{p})$ and the action of $W$ on $R\setminus R_L$ is defined in terms of the action of $G$ on $\mathfrak{g}/\mathfrak{p}$.  

To show that $T$ maps into the codomain as claimed, we need to show that $z(E)\subseteq v(R^-\setminus R_L^-)$ whenever $v$ is the minimal element of $zW_P$.  Note that $v=zu$ for some $u\in W_P$ and $u^{-1}(R^-\setminus R^-_L)=R^-\setminus R^-_L$.   Since $E\subseteq R^-\setminus R^-_L$, we have
$$z(E)=vu^{-1}(E)\subseteq vu^{-1}(R^-\setminus R^-_L)=v(R^-\setminus R^-_L).$$

Now we show by induction that every eventual Peterson translate of $w$  is of the form $T(z)$ for some $z\in W^Q$ with $z\leq w$.  The base case is where our eventual Peterson translate is $(w,w(R^-\setminus R^-_L)\cap R^+)$, which is $T(w)$.  Let $(a,M)$ be an eventual Peterson translate.  By the inductive hypothesis, $(a,M)=\tau_\alpha(u, v(E))$ for some $v\in W^Q$ and $\alpha\in \LInv^P(v)$, where $u$ is the minimal element of $vW_P$.  However, $\tau_\alpha(u,v(E))=(\widetilde{r_\alpha u}, r_\alpha(v(E))$.  Since $\alpha\in R^+\setminus u(R^-_L)$, we have that $\widetilde{r_\alpha u}$ is the minimal element in $r_\alpha v W_P$.  Furthermore, $r_\alpha v= z u$, where $z\in W^Q$ and $u\in W_Q$, and $u(E)=E$ as previously argued.  Also, since $\alpha\in \LInv(v)$, we have $r_\alpha v\leq v$ and thus $z\leq v$.  Hence, $\tau_\alpha(u,v(E))=T(z)$ with $z\in W^Q$ and $z\leq w$.

Finally we show by induction (on $\ell(w)-\ell(z)$) that $T(z)$ is an eventual Peterson translate.  We strengthen the induction hypothesis by not requiring that $z\in W^Q$.  Again the base case is $z=w$, where $T(w)$ is an eventual Peterson translate by definition.  Given $z\in W$ with $z<w$, by Lemma \ref{L:P-Bruhat_structure}, there exists $y$ with $z<y\leq w$ and $\alpha\in\LInv^P(y)$ such that $z=r_\alpha y$.  By the induction hypothesis, $T(y)$ is an eventual Peterson translate, and $T(z)=\tau_\alpha(T(y))$.  Hence $T(z)$ is also an eventual Peterson translate.
\end{proof}

\section{Specifics for Grassmannians}\label{S:type_A_case}

In type A, every maximal parabolic $P\subseteq G$ is cominuscule.   If the simple root associated to $P$ is $\alpha_k$, then the variety $G/P$ is the Grassmannian parameterizing $k$-planes in some fixed $\mathbb{C}^n$.  In this section, we give a combinatorial description of the Nash blow-up as a configuration space of subspaces of $\mathbb{C}^n$ and describe when the Nash blow-up is actually smooth.   We also show how the Nash blow-up relates to the Zelevinsky resolution of Grassmannian Schubert varieties.



\subsection{Schubert varieties as configuration spaces}

For $G=GL_n(\mathbb{C})$, the flag variety $G/B$ can be described as the space parameterizing all complete flags in $\mathbb{C}^n$.  Here, a {\bf complete flag} is a configuration $F_\bullet=(F_1\subset F_2\subset\cdots\subset F_{n-1})$ where each $F_i\subseteq \mathbb{C}^n$ is an $i$-dimensional subspace.  Let $\Fl(n)$ denote the variety of complete flags in $\mathbb{C}^n$.  Similarly, given a parabolic $P$, the variety $G/P$ parameterizes partial flags
$F_\bullet=(F_{a_1}\subset \cdots\subset F_{a_m})$, where $\dim F_{a_i}=a_i$, and $a_1<\cdots<a_m$ are the indices such that the simple reflection $s_{a_i}\notin W_P$.  Let $\Fla$ be the variety of partial flags in $\mathbb{C}^n$ corresponding to the sequence $\textbf{a}=(a_1<\cdots<a_m)$ (equivalently, corresponding to the parabolic $P$).


The Weyl group $W$ is the symmetric group $S_n$, and a Schubert variety can be defined in terms of intersection conditions on the flags its points represent.  Fix an ordered basis $e_1,\ldots,e_n\in\mathbb{C}^n$, and let $E_i=\langle e_1,\ldots,e_i\rangle$.  Let $w=w(1)\cdots w(n)\in S_n$, and for each $i,j$ such that $1\leq i\leq n$ and $1\leq j\leq m$, let the {\bf rank number} be $$r_{i,a_j}(wW_P):=\#\{k\leq a_j\mid w(k)\leq i\}.$$  Note these numbers are independent of the choice of coset representative in $wW_P$.  As a configuration space, the Schubert variety

$$X_w^P=\{F_\bullet\in \Fla\mid \dim(E_i\cap F_{a_j})\geq r_{i,a_j}(wW_P)\}.$$


Some of these intersection conditions will imply others, so not all of the given intersection conditions are needed to define the Schubert variety.  For any $w\in S_n$, define the {\bf coessential set} as
$$\Coess(w):=\{(p,q)\mid w^{-1}(p)\leq q<w^{-1}(p+1)\quad\text{and}\quad w(q)\leq p<w(q+1)\}.$$

Up to a change of convention, the following is a lemma of Fulton \cite[Lemmas 3.10 and 3.14]{Fu92}.  


\begin{lemma}\label{L:Coess_Full}
Let $w\in S_n$.  Then the Schubert variety
$$X_w^B=\{F_\bullet\in \Fl(n)\mid \dim(E_p\cap F_q)\geq r_{p,q}(w),\ (p,q)\in\Coess(w)\}.$$
Furthermore, $\Coess(w)$ is the minimal set of intersection conditions defining $X_w^B$.
\end{lemma}


Given $w\in S_n$ and a parabolic subgroup $P$, let $v_P(w)$ be the {\em maximal} coset representative in $wW_P$.  If $s_q\in W_P$, then $v_P(w)(q)>v_P(w)(q+1)$.  This inequality implies $s_q\notin W_P$ for every $(p,q)\in\Coess(v_P(w))$.  Hence all the conditions imposed by $\Coess(v_P(w))$ relate to subspaces in a partial flag parameterized by $G/P$.  In particular, these conditions define $X_w^P$.  Note that the map $\pi: X^B_{v_P(w)}\rightarrow X_w^P$ is the $B/P$-fiber bundle where the the fiber $\pi^{-1}(x)$ is the collection of all the complete flags that complete the partial flag represented by $x\in X_w^P$.  We also note that if $w$ is a minimal coset representative, then $v_P(w)=ww_0^P$ where $w_0^P$ denotes the longest element of $W_P$.


\subsection{Grassmannian Schubert varieties and their Nash blow-ups}

We now describe the coessential set of $v_P(w)$ when $w$ is a Grassmannian permutation with unique descent at $k$ and $P$ is the maximal parabolic generated by $\Delta_{L}=\{\alpha_i\mid i\neq k\}$.  We denote the corresponding Grassmannian by  $\Gr(k,n)$.

\begin{lemma}\label{L:Coess_Grass}
Let $w\in W^P$ be a Grassmannian permutation with unique descent at $k$.  Then
$$\Coess(v_P(w))=\{(i,k)\mid  \ w^{-1}(i)\leq k<w^{-1}(i+1)\}.$$
Furthermore, $r_{i,k}(v_P(w))=w^{-1}(i)$.  Hence
$$X_w^P=\{V\in \Gr(k,n)\mid \dim(V\cap E_i)=w^{-1}(i),\ (i,k)\in \Coess(v_P(w))\}.$$
\end{lemma}

\begin{proof}
Note that if either $w(q)<w(q+1)$ or $v_P(w)(q)>v_P(w)(q+1)$, then $q=k$ since $w$ is a Grassmannian permutation.  We have
$$v_P(w)^{-1}(i)\leq k<v_P(w)^{-1}(i+1)\quad \mbox{if and only if}\quad w^{-1}(i)\leq k<w^{-1}(i+1).$$
In addition, any $i$ satisfying $v_P(w)^{-1}(i)\leq k<v_P(w)^{-1}(i+1)$ automatically satisfies $v_P(w)(k)\leq i<v_P(w)(k+1)$, since $v_P(w)$ has a descent at every index between $v_P(w)^{-1}(i)$ and $k$ and between $k+1$ and $v_P(w)^{-1}(i+1)$.

Note $r_{i,k}(v_P(w))=r_{i,k}(w)=w^{-1}(i)$.  The last statement follows by the discussion following Lemma \ref{L:Coess_Full}.



\end{proof}

For convenience, we let $p_1<\cdots< p_m$ denote the first factors in the coessential set, so $\Coess(v_P(w))=\{(p_i,k)\mid 1\leq i\leq m\}$.  We also let the rank numbers $r_i=r_{p_i,k}(v_P(w))=w^{-1}(p_i)$.  Recall from Theorem \ref{T:main1} that the Nash blow-up of $X_w^P$ is isomorphic to the Schubert variety $X_w^Q$ where $Q$ is the parabolic subgroup generated by the simple roots $\Delta_w\subseteq\Delta_{L}$.  In the case of Grassmannian Schubert varieties, we have that
$$\Delta_w=\{\alpha_j\mid w(j+1)=w(j)+1,\ j\neq k\}.$$
Hence, if $w$ is not the identity, then $\alpha_j\notin \Delta_w$ if and only if either:
\begin{enumerate}
\item  $w(j)\leq k$ and $w(j+1)>k$ or
\item $w(j)>k$ and $w(j+1)\leq k$.
\end{enumerate}
We will see that these inequalities mean either $j=w^{-1}(p_i)=r_i$ for some $p_i$ or $j=w^{-1}(p_i+1)-1=k+p_i-r_i$.


\begin{lemma}\label{L:Coess_Nash}
Let $w\in W^P$ be a Grassmannian permutation with unique descent at $k$, and let $Q$ be generated by $\Delta_w$ as described in Theorem \ref{T:main1}.  If $w$ is not the identity, then
$$\Coess(v_Q(w))=A\cup B,$$
where
$$A=\{(i,w^{-1}(i))\mid w^{-1}(i)<k< w^{-1}(i+1)\}$$
and
$$B=\{(i, w^{-1}(i+1)-1) \mid w^{-1}(i)< k+1<w^{-1}(i+1)\}.$$
In particular,
$$A=\left\{(p_i,r_i)\ \text{such that}\ \begin{cases}1\leq i\leq m & \text{if}\ w(k)=n \\ 1\leq i\leq m-1& \text{if}\ w(k)<n\end{cases}\right\}$$
and
$$B=\left\{(p_i,k+p_i-r_i)\ \text{such that}\ \begin{cases}1\leq i\leq m & \text{if}\ w(k+1)=1 \\ 2\leq i\leq m & \text{if}\ w(k+1)>1\end{cases}\right\}.$$
Furthermore, $r_{p_i,r_i}(v_Q(w))=r_i$, and $r_{p_i,k+p_i-r_i}(v_Q(w))=p_i$.


\end{lemma}


\begin{proof}
Let $q<k$.  Then $v_Q(w)(q)<v_Q(w)(q+1)$ if and only if $w(q+1)\neq w(q)+1$.  Suppose $q$ satisfies either condition.  If $v_Q(w)^{-1}(p)\leq q <v_Q(w)^{-1}(p+1)$, then $p=w(q)$, since $v_Q(w)(j)\leq w(q)$ for all $j\leq q$.  Moreover, $v_Q(w)^{-1}(i+1)\leq q$ for all $i$ with $v_Q(w)^{-1}(q)\leq i<p$.  Since $w(q+1)\neq w(q)+1$, we have $w^{-1}(w(q)+1)>k>q$.  Hence for $q<k$, $(p,q)\in \Coess(v_Q(w))$ if and only if $(p,q)\in A$.

Since $w$ is not the identity and $w\in W^P$, we have $w(k)>w(k+1)$ and $v_Q(w)(k)>v_Q(w)(k+1)$.  Hence, $\Coess(v_Q(w))$ contains no elements $(a,b)$ where $b=k$.

Let $t>k$.  Similarly, we have $v_Q(w)(t)<v_Q(w)(t+1)$ if and only if $w(t+1)\neq w(t)+1$.  Suppose $t$ satisfies either condition.  If $v_Q(w)^{-1}(s)\leq t<v_Q(w)^{-1}(s+1)$, then $s+1=w(t+1)$, since $v_Q(w)(j)\geq w(t+1)$ for all $j>t$.  Moreover $v_Q(w)^{-1}(i)>t$ for all $i$ with $s+1\leq i\leq v_Q(w)(t+1)$.  Since $w(t+1)\neq w(t)+1$, we have $w^{-1}(w(t+1)-1))\leq k<t$.  Hence for $t>k$, $(s,t)\in\Coess(v_Q(w))$ if and only if $(s,t)\in B$.

If $w(k)=n$, then $(w^{-1}(k),k)\not\in\Coess(v_P(w))$, so $w^{-1}(p_i)<k$ and $(p_i,r_i)\in A$ for all $i$.  Otherwise, $w^{-1}(k)=p_m$, so $(p_m,r_m)\not\in A$.  Similarly, if $w(k+1)=1$, then $k+1<w^{-1}(p_i+1)$ and $(p_i,w^{-1}(p_i+1)-1\in B$ for all $i$.  Otherwise, $w^{-1}(p_1+1)=k+1$ and $(p_1,w^{-1}(p_1+1)-1)\not\in B$.  Also note that $w^{-1}(p_i+1)-1=k+p_i-r_i$, since $k<j<w^{-1}(p_i+1)$ if and only if $j>k$ and $w(j)\leq p_i$, and there are exactly $p_i$ indices such that $w(j)<p_i$, with $r_i$ of those indices being less than or equal to $k$.

For $(p,q)\in A$, we have that $w(j)\leq p$ for all $j\leq q$, and since $\alpha_q\not\in \Delta_w$, we also have $v_Q(w)(j)\leq p$ for all $j\leq q$.  Hence $r_{p,q}(v_Q(w))=q$.  For $(s,t)\in B$, we have $w(j)>s$ for all $j>t$ and $v_Q(w)(j)>s$ for all $j>t$.  Hence, if $v_Q(w)(j)\leq s$, then $j\leq t$.  Thus $r_{s,t}(v_Q(w))=s$.
\end{proof}

We now describe the Nash blow-up of a Grassmannian Schubert variety as a configuration space using Theorem \ref{T:main1} and Lemma \ref{L:Coess_Nash}.  For simplicity, we will assume $w\in W^P$ satisfies $w(k)=n$ and $w(k+1)=1$.

\begin{corollary}\label{C:type_A_Nashblowup}
Let $w\in W^P$ be a Grassmannian permutation with unique descent at $k$, coessential set $$\Coess(v_P(w))=\{(p_i,k)\mid 1\leq i\leq m\},$$ and rank numbers $r_i=r_{p_i,k}(v_P(w))=w^{-1}(p_i)$.
Further assume that $w(k)=n$ and $w(k+1)=1$.  Then the Nash blow-up of the Schubert variety $X_w^P$ is isomorphic to
$$X_w^Q=\{F_\bullet\in\Fl(r_1,\ldots,r_m,k,k+p_1-r_1,\ldots,k+p_m-r_m)\mid F_{r_i}\subseteq E_{p_i}\subseteq F_{k+p_i-r_i}\}.$$
\end{corollary}


The conditions that $w(k)=n$ and $w(k+1)=1$ in the corollary above are minor.  Indeed, if $w(k)<n$, then $r_m=k$ and hence $F_{r_m}=F_k\subset F_{p_m}=E_{p_m}$.  In this case, the condition that $F_{r_m}\subseteq E_{p_m}$ is not in the coessential set of $X_w^Q$, as it is forced by the condition $F_{p_m}=E_{p_m}$.  Note also that there is no choice for $F_{p_m}$.  If $w(k+1)>1$, then $p_1=r_1$ and $k+p_1-r_1=k$.  In this case, we require that $E_{p_1}=F_{p_1}$, and the condition that $E_{p_1}\subseteq F_k=F_{k+p_1-r_1}$ is not in the coessential set.

It is easy to see from this description that, for a generic point in $X_w^P$, the fiber of the Nash blow-up is a single point.  If $V\in X_w^P$ represents a generic point, then we will have $\dim(E_{p_i}\cap V)=r_i$ for all $i$, which implies $\dim(E_{p_i}+V)=k+p_i-r_i$.  Hence the unique point in the fiber will have $F_{r_i}=E_{p_i}\cap V$ and $F_{k+p_i-r_i}=E_{p_i}+V$.

It is well known that Schubert varieties in the Grassmannian $\Gr(k,n)$ correspond to partitions whose Young diagram is contained in a $k\times (n-k)$ rectangle.  We now show how the combinatorial data of the Nash blow-up can be read from the partition.  For any Grassmannian permutation $w\in W^P$, define the partition $\lambda(w)=(\lambda_1\geq\cdots\geq \lambda_k)$ by setting $\lambda_{k-i+1}:=w(i)-i$ for all $i\leq k$.  (Note that $\lambda(w)$ can have parts that are zero.)  We say an index $c$ is an {\bf inner corner} of $\lambda(w)$ if $\lambda_c>\lambda_{c+1}$.  By convention, $c=0$ is an inner corner if $\lambda_1<n-k$.  Rephrasing Lemma \ref{L:Coess_Grass} gives
$$\Coess(v_P(w))=\{(k-c+\lambda_{c+1},k)\mid c\mbox{ is an inner corner of $\lambda(w)$}\},$$
and $$r_{k-c+\lambda_{c+1},k}=k-c.$$
Each inner corner $c$ of $\lambda(w)$ corresponds to a coessential condition on the Nash blow-up of $X_w^P$.  In particular, the Schubert conditions describing $X_w^Q$ in Corollary \ref{C:type_A_Nashblowup} become
$$F_{k-c}\subseteq E_{k-c+\lambda_{c+1}}\subseteq F_{k+\lambda_{c+1}}.$$
In the next example, we demonstrate how the Young diagram of $\lambda(w)$ makes the data on the Nash blow-up relatively easy to calculate.




\begin{example}\label{Ex:431}
Let $n=8$, $k=3$ and $w=25713468$.  In this example, the coessential set $$\Coess(v_P(w))=\{(2,3), (5,3), (7,3)\},$$ which correspond to Schubert intersection conditions
$$\dim(E_2\cap V)\geq 1,\ \dim(E_5\cap V)\geq 2,\ \dim(E_7\cap V)\geq 3$$
for $V\in X_w^P\subseteq \Gr(3,8)$.

The partition $\lambda(w)=(4,3,1)$ has three inner corners at $c=2,1,0$ which correspond respectively to $(2,3), (5,3), (7,3)\in \Coess(v_P(w))$.  We label the lower boundary of the Young diagram of $\lambda(w)$ (given in Russian notation) with the indices $1,2,\ldots, 7$ and calculate the Nash blow-up by projecting from the inner corners to the lower boundary as follows:

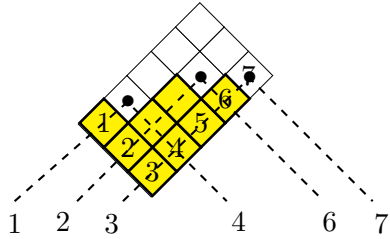
\YRussian
$$\begin{tikzpicture}
\draw (0.04,0.35) node {\gyoung(;;;;;7,;;;;;,;;;;;)};
\Yfillcolour{yellow}
\Ylinethick{1pt}
\draw (0,0) node {\gyoung(3;4;5;6:\bullet,2;;:\bullet,1:\bullet)};
\draw[thick, dashed] (-2.10,-1)--(-0.60,0.32)--(0.72,-1);
\draw[thick, dashed] (-1.27,-1)--(0.37,0.64)--(2.01,-1);
\draw[thick, dashed] (-0.66,-1)--(1.02,0.64)--(2.66,-1);
\draw (-2.1,-1.3) node {$ 1$};\draw (-1.455,-1.27) node {$2$};\draw (-.81,-1.3) node {$3$};
\draw (0.88,-1.27) node {$4$};\draw (2.09,-1.27) node {$6$};\draw (2.775,-1.3) node {$7$};
\end{tikzpicture}$$

In this case, $\Delta_w=\{\alpha_5\}$, and $G/Q$ is the partial flag variety $\Fl(1,2,3,4,6,7)$.  By Corollary \ref{C:type_A_Nashblowup}, the Nash blow-up of $X_w^P$ is isomorphic to the Schubert variety $X_w^Q$ given by all partial flags $F_\bullet\in\Fl(1,2,3,4,6,7)$ satisfying
$$F_{1}\subseteq E_{2}\subseteq F_{4},\qquad F_{2}\subseteq E_{5}\subseteq F_{6},\qquad F_{3}\subseteq E_{7}\subseteq F_{7}.$$
Observe that we are in the degenerate case where $w(k)<n$.

\end{example}

\subsection{Smoothness of Nash blow-ups}

We can determine when $X_w^P$ has a smooth Nash blow-up (or, in other words, when $X_w^Q$ is smooth) using the work of Gasharov and Reiner in \cite{GR02}.

Given $w\in S_n$ and a parabolic $P$, a Schubert variety $X_w^P$ is {\bf defined by inclusions} if $r_{p,q}(v_P(w))=\min(p,q)$ for all $(p,q)\in\Coess(v_P(w))$.  This condition is equivalent to stating that all the essential conditions defining $X_w^P$ are of the form $E_p\subseteq F_q$ or $F_q\subseteq E_p$.  Note that, by the last statement of Lemma~\ref{L:Coess_Nash}, $X_w^Q$ is always defined by inclusions.

A permutation $w$ is {\bf covexillary} if we do not have both $p<p'$ and $q>q'$ for all pairs of boxes $(p,q),(p',q')\in\Coess(w)$.  Equivalently, if $(p,q),(p',q')\in\Coess(w)$ with $p<p'$, then $q\leq q'$, and if $q>q'$, then $p\geq p'$.

The following is \cite[Theorem 1.1]{GR02} due to Gasharov and Reiner.

\begin{theorem}\label{T:GR_thoerem}
Suppose $X_w^Q$ is defined by inclusions.  Then $X_w^Q$ is smooth if and only if $v_Q(w)$ is covexillary.
\end{theorem}

Given a permutation $w$, we say the box $(p,q)\in \Coess(w)$ is an {\bf inclusion box} if $r_w(p,q)=\min(p,q)$.  In the case where $w$ is Grassmannian with unique descent at $k$, the following are true:
\begin{enumerate}
\item  $(p_1,k)$ is an inclusion box of $\Coess(v_P(w))$ if and only if $w(k+1)\neq 1$,
\item  $(p_m,k)$ is an inclusion box if and only if $w(k)\neq n$, and
\item  $(p_i,k)$ is not an inclusion box if $1<i<m$.
\end{enumerate}
In terms of the partition $\lambda(w)$, an inner corner $c$ gives an inclusion box if and only if $\lambda_{c+1}=0$ or $c=0$.  Note that in Example \ref{Ex:431}, the box $(7,3)$ is an inclusion box while $(2,3)$ and $(5,3)$ are not.

\begin{proposition}
Let $w\in W^P$ be a Grassmannian permutation.  The Nash blow-up of $X_w^P$ is smooth if and only if $\Coess(v_P(w))$ has at most one box that is not an inclusion box.
\end{proposition}

\begin{proof}
Suppose $(p_i,k)$ and $(p_j,k)$ are two non-inclusion boxes of $\Coess(v_P(w))$ with $i<j$ and hence $p_i<p_j$.  Then $(p_i,k+p_i-r_i)$ and $(p_j, r_j)$ are both in $\Coess(v_Q(w))$.  Since $p_i<p_j$ and $k+p_i-r_i>k>r_j$, we have that $v_Q(w)$ is not covexillary.  Therefore, $X_w^Q$ is not smooth by Theorem \ref{T:GR_thoerem}.

Now suppose $\Coess(v_P(w))$ has at most one non-inclusion box.  Let $(p,k)$ with $r=r_{p,k}(v_P(w))=w^{-1}(p)$ be the non-inclusion box if it exists.  If it exists, then $(p,r)$ and $(p,k+p-r)$ are boxes in $\Coess(v_Q(w))$.  If $w(k+1)\neq 1$, then $(p_1,p_1)$ is an inclusion box of $\Coess(v_Q(w))$, and if $w(k)\neq n$, then $(p_m, p_m)$ is also an inclusion box.  Note that $p_1<r<p$, $r<k+p-r$, and $p<k+p-r<k+p_m-r_m=p_m$ (whenever the relevant indices exist).  Hence $v_Q(w)$ is covexillary, and $X_w^Q$ is smooth by Theorem \ref{T:GR_thoerem}.
\end{proof}

Note that in Example \ref{Ex:431}, the Nash blow-up of $X_w^P$ is not smooth since $\Coess(v_P(w))$ has two boxes that are not inclusion boxes.

\subsection{Nash blow-ups and resolutions of Schubert varieties}

The configuration space description of the Nash blow-up shows that it is related to resolutions of singularities for Grassmannian Schubert varieties that have previously appeared in the literature.  Fix $w\in W^P$ with $\Coess(v_P(w))=\{(p_i,k)\mid 1\leq i\leq m\}$ and rank numbers $r_1,\ldots, r_m$.  Define
$$Z_w:=\{F_\bullet\in\Fl(r_1,\ldots,r_m,k)\mid F_{r_i}\subseteq E_{p_i} \mbox{ for all } i\},$$
and let $\pi_w: Z_w\rightarrow X_w^P$ be the map given by forgetting the data $F_{r_1},\ldots, F_{r_s}$.  We call the map $\pi_w$ the {\bf Cortez--Zelevinsky resolution} of $X_w^P$.  The variety $Z_w$ is smooth as it is an iterated fiber bundle with the fiber at each iteration being isomorphic to a Grassmannian.  It is one of a family of resolutions of a Schubert variety on the Grassmannian introduced by Zelevinsky in \cite{Ze83}; to be precise, it is the one given by choosing the peaks in order from left to right.  Cortez used this resolution to give an explicit description of the singular locus in \cite{Co03}, and Jones used it to give an explicit formula for the Chern--Mather and Chern--Schwartz--MacPherson classes of Schubert varieties on the Grassmannian in certain cases in \cite{Jo10}.

We also have a {\bf dual Cortez--Zelevinsky resolution} of $X_w^P$, given by
$$Z'_w=\{F_\bullet\in\Fl(k,k+p_1-r_1,\ldots,k+p_m-r_m)\mid E_{p_i}\subseteq F_{k+p_i-r_i} \mbox{ for all } i\}.$$
This is the resolution given by Zelevinsky's construction if one orders the peaks from right to left.  The next result follows immediately from Corollary \ref{C:type_A_Nashblowup}.

\begin{corollary}
Let $X_w^P$ be a Grassmannian Schubert variety.  Then the Nash blow-up of $X_w^P$ is isomorphic to $Z_w\times_{X_w^P} Z'_w$.
\end{corollary}

We now describe a conjectural extension of our results for type A Schubert varieties associated to covexillary permutations.  By definition, if $w$ is covexillary, we can write $$\Coess(w)=\{(p_i,q_i)\mid 1\leq i\leq m\}$$ where $p_1\leq\cdots\leq p_m$ and $q_1\leq\cdots\leq q_m$.  Let $r_i=r_{p_i,q_i}(w)$ be the rank numbers of $w$.
Let $P$ be the largest standard parabolic subgroup such that $w$ is a maximal coset representative for $wW_P$.  To be precise, $P$ is generated by $\{\alpha_j\mid w(j)>w(j+1)\}=\{\alpha_j\mid j\neq q_i \mbox{ for all } i\}$.  Note that $P$ may be a non-maximal parabolic subgroup and hence not cominuscule.  The Cortez--Zelevinsky resolution was defined for covexillary Schubert varieties by Cortez~\cite{Co03}, with
$$Z_w=\{(F_\bullet,V_\bullet)\in\Fl(r_1,\ldots,r_m)\times\Fl(q_1,\ldots,q_m)\mid F_{r_i}\subseteq (E_{p_i}\cap V_{q_i})\}.$$
The resolution map $\pi_w: Z_w\rightarrow X_w^P$ is given by projection to the second factor.
The dual Cortez--Zelevinsky resolution can also be described as
$$Z'_w=\{(F_\bullet,V_\bullet)\in\Fl(q_1+p_1-r_1,\ldots,q_m+p_m-r_m)\times\Fl(q_1,\ldots,q_m)\mid F_{q_i+p_i-r_i}\supseteq (E_{p_i}+ V_{q_i})\}.$$

\begin{conjecture}
Let $w$ be a covexillary permutation with parabolic $P$ defined as above.  The Nash blow-up of $X_w^P$ is isomorphic to $Z_w\times_{X_w^P} Z'_w$.
\end{conjecture}

We have verified for small to medium cases that, for each torus fixed point $vP\in X_w^P$ of the Schubert variety, the number of torus-fixed points in the fiber over $vP$ of $Z_w\times_{X_w^P} Z'_w$ equals the number of eventual Peterson translates of $w$ of the form $(v,N)$ for some subspace $N$ of $T_v(X_w^P)$.

\bibliographystyle{amsalpha}
\bibliography{nash}

\providecommand{\bysame}{\leavevmode\hbox to3em{\hrulefill}\thinspace}
\providecommand{\MR}{\relax\ifhmode\unskip\space\fi MR }
\providecommand{\MRhref}[2]{%
  \href{http://www.ams.org/mathscinet-getitem?mr=#1}{#2}
}
\providecommand{\href}[2]{#2}
\begin{thebibliography}{RRS92}

\bibitem[AM09]{AM09}
Paolo Aluffi and Leonardo~Constantin Mihalcea, \emph{Chern classes of
  {S}chubert cells and varieties}, J. Algebraic Geom. \textbf{18} (2009),
  no.~1, 63--100. \MR{2448279}

\bibitem[AM16]{AM16}
Paolo Aluffi and Leonardo~C. Mihalcea, \emph{Chern-{S}chwartz-{M}ac{P}herson
  classes for {S}chubert cells in flag manifolds}, Compos. Math. \textbf{152}
  (2016), no.~12, 2603--2625. \MR{3594289}

\bibitem[AMSS]{AMSS17}
Paolo Aluffi, Leonardo~C. Mihalcea, Joerg Schuermann, and Changjian Su,
  \emph{Shadows of characteristic cycles, verma modules, and positivity of
  chern-schwartz-macpherson classes of schubert cells}, preprint,
  arXiv:1709.08697.

\bibitem[BB03]{BB03}
Sara Billey and Tom Braden, \emph{Lower bounds for {K}azhdan-{L}usztig
  polynomials from patterns}, Transform. Groups \textbf{8} (2003), no.~4,
  321--332. \MR{2015254}

\bibitem[BB05]{BB}
Anders Bj\"orner and Francesco Brenti, \emph{Combinatorics of {C}oxeter
  groups}, Graduate Texts in Mathematics, vol. 231, Springer, New York, 2005.
  \MR{2133266}

\bibitem[BP99]{BP99}
Michel Brion and Patrick Polo, \emph{Generic singularities of certain
  {S}chubert varieties}, Math. Z. \textbf{231} (1999), no.~2, 301--324.
  \MR{1703350}

\bibitem[BP05]{BP05}
Sara Billey and Alexander Postnikov, \emph{Smoothness of {S}chubert varieties
  via patterns in root subsystems}, Adv. in Appl. Math. \textbf{34} (2005),
  no.~3, 447--466. \MR{2123545}

\bibitem[BS98]{BS98}
Nantel Bergeron and Frank Sottile, \emph{Schubert polynomials, the {B}ruhat
  order, and the geometry of flag manifolds}, Duke Math. J. \textbf{95} (1998),
  no.~2, 373--423. \MR{1652021}

\bibitem[BS02]{BS02}
\bysame, \emph{A {P}ieri-type formula for isotropic flag manifolds}, Trans.
  Amer. Math. Soc. \textbf{354} (2002), no.~7, 2659--2705. \MR{1895198}

\bibitem[CK03]{CK03}
James~B. Carrell and Jochen Kuttler, \emph{Smooth points of {$T$}-stable
  varieties in {$G/B$} and the {P}eterson map}, Invent. Math. \textbf{151}
  (2003), no.~2, 353--379. \MR{1953262}

\bibitem[Cor03]{Co03}
Aur\'elie Cortez, \emph{Singularit\'es g\'en\'eriques et quasi-r\'esolutions
  des vari\'et\'es de {S}chubert pour le groupe lin\'eaire}, Adv. Math.
  \textbf{178} (2003), no.~2, 396--445. \MR{1994224}

\bibitem[Ful92]{Fu92}
William Fulton, \emph{Flags, {S}chubert polynomials, degeneracy loci, and
  determinantal formulas}, Duke Math. J. \textbf{65} (1992), no.~3, 381--420.
  \MR{1154177}

\bibitem[GR02]{GR02}
V.~Gasharov and V.~Reiner, \emph{Cohomology of smooth {S}chubert varieties in
  partial flag manifolds}, J. London Math. Soc. (2) \textbf{66} (2002), no.~3,
  550--562. \MR{1934291}

\bibitem[Jon10]{Jo10}
Benjamin~F. Jones, \emph{Singular {C}hern classes of {S}chubert varieties via
  small resolution}, Int. Math. Res. Not. IMRN (2010), no.~8, 1371--1416.
  \MR{2628830}

\bibitem[Mac74]{Ma74}
R.~D. MacPherson, \emph{Chern classes for singular algebraic varieties}, Ann.
  of Math. (2) \textbf{100} (1974), 423--432. \MR{0361141}

\bibitem[Nob75]{No75}
A.~Nobile, \emph{Some properties of the {N}ash blowing-up}, Pacific J. Math.
  \textbf{60} (1975), no.~1, 297--305. \MR{0409462}

\bibitem[Per09]{Pe09}
Nicolas Perrin, \emph{The {G}orenstein locus of minuscule {S}chubert
  varieties}, Adv. Math. \textbf{220} (2009), no.~2, 505--522. \MR{2466424}

\bibitem[Rob14]{Ro14}
Colleen Robles, \emph{Singular loci of cominuscule {S}chubert varieties}, J.
  Pure Appl. Algebra \textbf{218} (2014), no.~4, 745--759. \MR{3133705}

\bibitem[RRS92]{RRS92}
Roger Richardson, Gerhard R\"ohrle, and Robert Steinberg, \emph{Parabolic
  subgroups with abelian unipotent radical}, Invent. Math. \textbf{110} (1992),
  no.~3, 649--671. \MR{1189494}

\bibitem[RS16]{RS16}
Edward Richmond and William Slofstra, \emph{Billey-{P}ostnikov decompositions
  and the fibre bundle structure of {S}chubert varieties}, Math. Ann.
  \textbf{366} (2016), no.~1-2, 31--55. \MR{3552231}

\bibitem[Zel83]{Ze83}
A.~V. Zelevinski\u\i, \emph{Small resolutions of singularities of {S}chubert
  varieties}, Funktsional. Anal. i Prilozhen. \textbf{17} (1983), no.~2,
  75--77. \MR{705051}

\end{thebibliography}

\end{document}